\documentclass[a4paper,oneside]{article}

\usepackage{amsmath}%
\usepackage{amssymb}%
\usepackage[latin1]{inputenc}

\usepackage{geometry}
\newtheorem{theorem}{Theorem}[section]

\newtheorem{definition}[theorem]{Definition}

\newtheorem{lemma}[theorem]{Lemma}

\newtheorem{proposition}[theorem]{Proposition}

\newenvironment{proof}[1][Proof]{\textbf{#1.} }{\ \rule{0.5em}{0.5em}}

\newcommand{\refeqn}[1]{(\ref{#1})}

\newcommand{\cinf}[0]{C^{\infty}}

\newcommand{\matr}[0]{\operatorname{Mat}}
\newcommand{\spann}[0]{\operatorname{span}}
\newcommand{\Iso}[0]{\operatorname{Iso}}

\begin{document}

\title{{\bf Symmetries of Stochastic Differential Equations: a geometric approach}}
\author{ Francesco C. De Vecchi*, Paola Morando** \\
and Stefania Ugolini*\\
*Dip. di Matematica, Universit\`a degli Studi di Milano, \\
via Saldini 50, Milano;\\
**DISAA, Universit\`a degli Studi di Milano, \\
via Celoria 2, Milano;\\
francesco.devecchi@unimi.it, paola.morando@unimi.it, \\
stefania.ugolini@unimi.it }
\date{}

\maketitle
\begin{abstract}
A new notion of stochastic transformation  is proposed and applied to the study of both weak and strong  symmetries of stochastic differential equations (SDEs). The correspondence between an algebra of weak symmetries  for a given SDE and  an algebra of strong  symmetries for a modified SDE is proved under suitable regularity assumptions.
This general approach is applied to a stochastic version of a two dimensional  symmetric ordinary differential equation and to the case of two dimensional Brownian motion.
\end{abstract}

\textbf{Keywords}: Stochastic Differential Equations, Symmetry

\textbf{MSC numbers}: 60H10, 58D19

\section{Introduction}
The study of  symmetry properties of  ordinary (ODEs) and partial differential equations (PDEs) is a classical and well-established topic  (see \cite{Bluman-Kumei,Gaeta,Olver,Olver2,Stephani})  and provides a powerful tool for both the explicit computation of  solutions to the equations and a better understanding of their qualitative behavior. On the other hand an analogous theory of  symmetries of stochastic differential equations (SDEs) or, more generally, of diffusion processes has been developed only in recent years. In the literature we can distinguish three different approaches.
The first one goes back to 1995 and is due to Cohen De Lara (\cite{DeLara1}), who studied the symmetries of the infinitesimal generator of a diffusion process on a manifold, starting from pioneer works on symmetries of Markov processes (\cite{Glover1},\cite{Glover2}). A recent similar method is introduced in \cite{DeVecchi1}, where a geometric definition of  symmetry for diffusion processes is proposed, in the framework  of the second order geometry developed by Schwartz, Emery and  Meyer (\cite{ Schwartz},\cite{Emery}) and a geometric reformulation of the associated martingale problem is provided.\\
In the second approach  a symmetry of a SDE is defined as a suitable transformation preserving the solution to the  SDE, in complete analogy with the deterministic case  of ODEs by Gaeta (\cite{Gaeta1},\cite{Gaeta2}). A review of the results in this direction up to 2010 can be found in \cite{Grigoriev} (Chapter 5). Finally a very general and elegant approach for solving a SDE  via symmetries in a dynamical perspective, founded on quantum mechanical considerations, is due to Zambrini et al. (\cite{Lescot}).\\
In this paper we follow the second approach, showing how some relevant results of the first one can be recovered in this setting. In particular
we consider pairs ($X$,$W$), where $X$ is a continuous stochastic process and  $W$ is an $m$-dimensional Brownian motion, and we define a class of  transformations characterized by three geometrical objects: a diffeomorphism $\Phi$  describing the transformation of the state variable $X$, a matrix valued function $B$  inducing a general state dependent rotation of the Brownian motion $W$ and a density function $\eta$ representing a random time change of the process ($X$,$W$). We call the triad $(\Phi,B,\eta)$ a general \emph{(finite) stochastic transformation}. \\
Each part of this general transformation $(\Phi,B,\eta)$ has been already considered in some of the previous references: random time change has been used for example in \cite{DeLara1,Grigoriev,Srihirun},  rotation of Brownian motion with a constant matrix $B$ is the \emph{$W$-symmetry} introduced in \cite{Gaeta2} and in \cite{Glover1}  the Authors, facing the problem of generating transformations $\Phi$ preserving the Markov property of a continuous process, acknowledge the deep connection between $\Phi$ and $\eta$ although  they find the candidate function $\Phi$ by exploiting the symmetries of the original process and then construct the appropriate time scale to go with it.

In this paper, for the first time, these three transformations are considered all at once and a geometrical description of general stochastic transformations in terms of isomorphisms of (trivial) principal bundles is provided. With any triad $T=(\Phi,B,\eta)$ we associate a process transformation $P_T$ and a SDE transformation $E_T$ with a relevant probabilistic meaning (Section \ref{section_transformation}). Moreover the set of these general stochastic transformations forms an infinite dimensional Lie group with respect to a composition rule with a deep probabilistic consistency (Section \ref{section_probability_meaning}). In this setting it is natural to consider  the Lie algebra associated with  this group, whose elements $(Y,C,\tau)$ represent  infinitesimal stochastic transformations. Since in the following the transformations $(\Phi,I_m,1)$, involving only state transformation, play an important role, we call them \emph{strong stochastic transformations}. The corresponding \emph{strong  infinitesimal stochastic transformations} have the form $(Y,0,0)$.\\
Given  a SDE with coefficients
$(\mu,\sigma)$  we consider  (weak) solutions   ($X$,$W$) to $(\mu,\sigma)$ and
we call \emph{(weak) symmetry} of $(\mu,\sigma)$  the general stochastic transformation transforming  any (weak) solution $(X,W)$ to $(\mu,\sigma)$ in another (weak) solution ($X'$,$W'$) to $(\mu,\sigma)$, where the relationship between the two Brownian motions $W$ and $W'$ is explicitly available. If the stochastic transformation is of the form $(\Phi,I_m,1)$ we call it a \emph{strong symmetry} of $(\mu,\sigma)$. We write the determining equations for the functions $\Phi,B,\eta$ (or for their infinitesimal generators $Y, C, \tau$) providing a necessary and sufficient condition for the stochastic transformation $(\Phi,B,\eta)$ to be a symmetry of $(\mu,\sigma)$. \\
This general  geometric approach  not only gives us a new perspective on the topic but also allows us to get  some new results. In particular  we prove that, given a SDE admitting  a Lie algebra of infinitesimal symmetries $(Y_i,C_i,\tau_i)$ satisfying suitable  non-degeneracy condition, it is possible to find  a stochastic transformation such that  the transformed  SDE  admits a Lie algebra of strong infinitesimal symmetries of the form $(Y_i,0,0)$. This can be very useful as strong  symmetries are rightly well acknowledged in literature mainly for their connection   with the existence of conservation laws (\cite{Albeverio}, \cite{Unal}) and with the corresponding reducibility properties of a SDE (see \cite{Lazaro_Ortega}).
We apply this general result to a SDE obtained as the (standard) stochastic perturbation of a  symmetric two-dimensional ODE and we explicitly compute  the finite stochastic transformation transforming weak symmetries of this SDE  into  strong symmetries of the corresponding transformed SDE.\\
Moreover we analyze, using this geometric approach,  the two dimensional Brownian motion proving that it admits an infinite number of (weak) symmetries.

The paper is organized as follows: in Section 2 we  introduce  the general idea of a finite stochastic transformation with a significant probabilistic foundation and we define the SDE and the process transformations corresponding to it. In Section 3 we provide  a geometrical description of stochastic transformations in terms of isomorphisms of principal bundles and we discuss the connection between the geometric and the probabilistic frameworks. In Section 4 the previous approach is applied to the study of symmetries of a SDE and  the determining equations are explicitly computed. Finally, in Section 5, the case of  a SDE obtained by stochastic perturbation of a  symmetric two-dimensional ODE and the two-dimensional Brownian motion are analyzed in detail.

\section{SDE Transformations: a probabilistic analysis}\label{section_transformation}

Let $M$ be an open subset of $\mathbb{R}^n$. We denote by
$x=(x^1,...,x^n)^T$ the standard Cartesian coordinate system
on $M$ and  by $\partial_i$ the derivative with respect $x^i$.
In the following $\cdot$ denotes the usual product between matrices.\\

Let us  consider all processes defined on a
probability space $(\Omega,\mathcal{F},\mathbb{P})$ and
denote by $\mathcal{F}_t \subset \mathcal{F}$ a filtration of
$\Omega$. Unless otherwise specified, we assume that the
stochastic processes are  adapted with respect to the filtration
$\mathcal{F}_t$.\\
If $X$ is a stochastic process on $M$ we denote by $X_t$ the value of
the process $X$ at time $t$ and  by $X^i$ the real processes defined as $X^i=x^i(X)$.\\
Let us consider a $m$-dimensional Brownian motion
$W=(W^1,...,W^m)=(W^{\alpha})$ and two smooth functions $\mu:M
\rightarrow \mathbb{R}^n$ and $\sigma:M \rightarrow \matr(n,m)$.

\begin{definition}
A stochastic process $X$ on $M$ and a $m$-dimensional Brownian motion $W$ (in short the process $(X,W)$) solves (in the weak sense) the SDE with coefficients
$\mu,\sigma$ until the stopping time
$\tau$ (or  shortly solves the SDE $(\mu,\sigma)$) if for
any $t \in \mathbb{R}_+$
$$X^i_{t \wedge \tau}-X^i_0=\int_0^{t \wedge \tau}{\mu^i(X_s)ds}+\int_0^{t \wedge \tau}{\sigma_{\alpha}^i(X_s)dW^{\alpha}_s}.$$
If $(X,W)$ solves the SDE $(\mu,\sigma)$ we write, as usual,
\begin{eqnarray*}
dX_t&=&\mu(X_t)dt+\sigma(X_t)\cdot dW_t\\
&=&\mu dt+\sigma \cdot dW_t.
\end{eqnarray*}
\end{definition}

The stopping time $\tau$ is strictly less then the explosion time of the SDE. When not strictly necessary, we  omit the stopping time $\tau$ from the definition of solution to a SDE.

In the following, in order to introduce the random time
change of a solution to a SDE,  we consider only autonomous SDEs. The generalization to nonautonomous SDEs will be considered  in a following paper.

\subsection{SDE Space Transformations}

If $A:M \rightarrow \matr(m,k)$ we write $A^l_r$ for the $l$-th row
and $r$-th column component of the matrix $A$ and identify
$\matr(k,1)$ with $\mathbb{R}^k$. Given a function $\Phi:M \rightarrow
\mathbb{R}^m$ we consider the smooth function
$D(\Phi):M \rightarrow \matr(n,m)$ defined by
$$D(\Phi)^l_i=\partial_i\Phi^l.$$

It is well known that to a SDE $(\mu,\sigma)$  it is possible to
associate a second order differential operator
$$L=A^{ij}\partial_i\partial_j+\mu^i\partial_i,$$
where  $A=\frac{1}{2}\sigma \cdot
\sigma^T$. The operator $L$ is called the infinitesimal
generator of the process and appears for example in the
following  important formula.

\begin{theorem}[Ito formula]
Let $(X,W)$ be a solution to  the SDE $(\mu,\sigma)$ and let $f:M \rightarrow
\mathbb{R}$ be a smooth function. Then  $F=f(X)$ satisfies
$$dF_t=L(f)(X_t)dt+D(f)(X_t) \cdot \sigma(X_t) \cdot dW_t.$$
\end{theorem}

\noindent By using the well-known Ito formula we can prove the following

\begin{proposition}\label{proposition_space_transformation}
Let us consider a diffeomorphism $\Phi:M \rightarrow M'$.
If $(X,W)$ is a solution to the SDE $(\mu,\sigma)$, then $(\Phi(X),W)$
is a solution to the SDE $(\mu',\sigma')$ where
\begin{eqnarray*}
\mu'&=&L(\Phi) \circ \Phi^{-1}\\
\sigma'&=&(D(\Phi) \cdot \sigma)\circ \Phi^{-1}.
\end{eqnarray*}
\end{proposition}

\begin{proof}
By using  the Ito formula, if $X'=\Phi(X)$, we have that
$X'^i=\Phi^i(X)$ and so
\begin{eqnarray*}
dX'^i_t&=&L(\Phi^i)(X_t)dt+D(\Phi^i)(X_t) \cdot \sigma(X_t) \cdot dW_t\\
&=&(L(\Phi^i) \circ \Phi^{-1})(X'_t)dt+(D(\Phi^i) \circ
\Phi^{-1})(X'_t)\cdot \sigma(\Phi^{-1}(X'_t)) \cdot dW_t.
\end{eqnarray*}
${}\hfill$\end{proof}

\subsection{SDE Random Time Transformations}

Let $\beta$ be a positive adapted stochastic process such that for any $\omega \in \Omega$ the function
$\beta(\omega):t \mapsto \beta_t(\omega)$ is continuous and strictly increasing. Define
$$\alpha_t=\inf\{s|\beta_s>t\},$$
where, as usual, $\inf(\mathbb{R}_+)=+\infty$.
The process $\alpha$ is an adapted process such that
$$\beta_{\alpha_t}=\alpha_{\beta_t}=t.$$
If $Y$ is a continuous stochastic process we define by $H_{\beta}(Y)$ the continuous stochastic process such that
$$H_{\beta}(Y)_t=Y_{\alpha_t}.$$
The process $H_{\beta}(Y)$ is an adapted process with respect the filtration $\mathcal{F}'_t=\mathcal{F}_{\alpha_t}$.
We restrict ourselves to absolute continuous time change. Given a strictly positive smooth
function $\eta:M \rightarrow \mathbb{R}_+$ and  a stochastic process $X$ defined until the
stopping time $\tau$, we consider   the process defined until
the stopping time $\tau$
$$\beta_{t \wedge \tau}=\int_0^{t \wedge \tau}{\eta(X_s)ds}.$$
If $X$ is a stochastic process, we denote by $H_{\eta}(X):=H_{\beta}(X)$ and it is easy to prove that
\begin{eqnarray*}
d(\alpha_t)&=&\frac{1}{\eta(H_{\eta}(X)_t)}dt.\\
\end{eqnarray*}

\begin{proposition}\label{proposition_time_transformation}
 Let $\eta:M \rightarrow \mathbb{R}_+$ be  a   smooth
function and $(X,W)$ be a solution to the SDE $(\mu,\sigma)$.   Then
$(H_{\eta}(X),H_{\eta}(W'))$, with
$$dW'_t=\sqrt{\eta(X_t)}dW_t,$$
is a solution to the SDE $(\mu',\sigma')$ where
\begin{eqnarray*}
\mu'&=&\frac{1}{\eta}\mu\\
\sigma'&=& \frac{1}{\sqrt{\eta}}\sigma\\
\end{eqnarray*}
\end{proposition}

The following theorem (see, e.g. \cite{Oksendal}, Corollary 8.5.5) expresses an important invariance property of Brownian
motion.
\begin{theorem}\label{theorem_time_change}
Let $X$ be a continuous  stochastic processes taking values in $M$ and
let $\eta,\alpha_{t'},\beta_t$ be as before. If   $W$ is an $m$-dimensional Brownian motion, denoting by $W'$ the stochastic process such that
$$dW'_t=\sqrt{\eta(X_t)}dW_t,$$
we have that $H_{\eta}(W')$ is an $m$-dimensional Brownian
motion.
\end{theorem}

\begin{proof}[Proof of Proposition \ref{proposition_time_transformation}]
Let $\tau$ be the stopping time related with the solution $(X,W)$ of
the SDE $(\mu,\sigma)$. Denoting by $\tau':=\beta_{\tau}$, we
prove that $(H_{\eta}(X),H_{\eta}(W'))$
is a solution to the SDE $(\mu',\sigma')$ until the stopping time
$\tau'$. In fact by definition
$X'_{t'}=X_{\alpha_{t'}}$ (with $t'=\beta_t$) and therefore
\begin{eqnarray*}
X_{\tau' \wedge t'}&=&X_{\tau \wedge \alpha_{t'}}\\
&=&\int_0^{\tau \wedge \alpha_{t'}}{\mu(X_s)ds+\sigma(X_s)\cdot dW_s}\\
&=&\int_{0}^{\beta_{\tau} \wedge \beta_{\alpha_{t'}}}{\mu(H_{\eta}(X)_s)d\alpha_s+\sigma(H_{\eta}(X)_s)\cdot d(H_{\eta}(W)_s)}\\
&=&\int_{0}^{\tau' \wedge
t'}{\frac{1}{\eta(X'_s)}\mu(X'_s)ds+\frac{1}{\sqrt{\eta(X'_s)}}\sigma(X'_s) \cdot d(H_{\eta}(W'))_s}.
\end{eqnarray*}
being $\beta_{\alpha_{t'}}=t'$.
${}\hfill$\end{proof}

\subsection{Transformations of  Brownian motion}

In the following if $Z^1,Z^2$ are two $L^2$ real semimartingales
we denote by $[Z^1,Z^2]$ the quadratic covariation between $Z^1$
and $Z^2$.  We recall the following

\begin{lemma}\label{theorem_quadratic_variation}
If $dZ^i_t=\alpha^i_{0,t} dt+\sum_{\beta=1}^{k}\alpha^i_{\beta,t}dW^{\beta_t}$ ($i=1,2$) for some real
stochastic processes $\alpha^i_{\beta}$, then
$$d[Z^1,Z^2]_t=\sum_{\beta=1}^k{\alpha^1_{\beta,t}\alpha^2_{\beta,t}dt}.$$
\end{lemma}
\begin{proof}
The proof is a consequence of Theorem 28.1 in \cite{Rogers_Williams}.
${}\hfill$\end{proof}

\begin{proposition}\label{proposition_Brownian_transformation}
Let $B:M \rightarrow SO(m)$ be a smooth function and $(X,W)$ be a solution to a SDE $(\mu,\sigma)$.
Then $(X,W')$, where
$$dW'_t=B(X_t) \cdot dW_t,$$
is a solution to the SDE $(\mu',\sigma')$
\begin{eqnarray*}
\mu'&=&\mu,\\
\sigma'&=&\sigma \cdot B^{-1}.\\
\end{eqnarray*}
\end{proposition}

\begin{proof}
The only thing to prove is that $W'$ is a Brownian motion. Indeed,
by the properties of Ito integral, we have
$$dX_t=\mu dt+\sigma \cdot dW_t=\mu dt+\sigma \cdot B^{-1} dW'_t=\mu'dt+\sigma' \cdot dW'_t.$$
First of all for any $\alpha$, $W'^{\alpha}$ is a local
martingale, because it is an Ito integral along the local
martingale $W^{\beta}$. On the other hand, by the properties of
the Ito integral, we have
\begin{eqnarray*}
[W'^{\alpha},W'^{\beta}]_t&=&\int_0^t{B^{\alpha}_{\gamma}B^{\beta}_{\delta}d[W^{\gamma},W^{\delta}]_s}\\
&=&\int_0^t{B^{\alpha}_{\gamma}B^{\beta}_{\delta}\delta^{\gamma\delta}ds}\\
&=&\int_0^t{\sum_{\gamma}B^{\alpha}_{\gamma}B^{\beta}_{\gamma}ds}\\
&=&\int_0^t{\delta^{\alpha\beta}ds}\\
&=&\delta^{\alpha\beta}t,
\end{eqnarray*}
where we use
$[W^{\gamma},W^{\delta}]=\delta^{\gamma\delta}s$  and
$B \cdot B^T=I$. The Levy characterization of Brownian motion (see,
for example, \cite{Rogers_Williams}, Chapter 4, Thm.(33.1)) ensures that $W'$ is a Brownian motion.
${}\hfill$\end{proof} \\

\subsection{Finite Stochastic Transformations}

In the following we consider two  open subsets $M',M''$  of $\mathbb{R}^n$
diffeomorphic to $M$, and we denote by $SO(m)$ the Lie
group of orthogonal matrices.

\begin{definition}\label{finitetrasformation}
Let $\Phi:M \rightarrow M'$ be a diffeomorphism, and let $B:M
\rightarrow SO(m)$ and $\eta:M \rightarrow \mathbb{R}_+$ be smooth
functions. We call the triad $T:=(\Phi,B,\eta)$  a (finite) stochastic transformation
 from $M$ onto $M'$ and we denote by $S_m(M,M')$ the set of all stochastic transformations from $M$ onto $M'$. If $T$ is of the form $T=(\Phi,I_m,1)$ we call $T$  a strong stochastic transformation and we denote the set of strong stochastic transformations by $SS_m(M,M')$.
\end{definition}

\begin{definition}\label{processtrasformation}
Let $T=(\Phi,B,\eta)$ be a stochastic transformation. If the pair $(X,W)$ is a continuous stochastic process, with   $X$ taking values on
$M$ and $W$ being an $m$-dimensional Brownian motion, we define the process $P_T(X,W)=(P_T(X),P_T(W))$ where $P_T(X)$ takes values on $M'$, given by
\begin{eqnarray*}
P_T(X)&=&\Phi(H_{\eta}(X)),\\
dW'_t&=&\sqrt{ \eta(X_t)} B(X_t) \cdot dW_t\\
P_T(W)&=&H_{\eta}(W').
\end{eqnarray*}
We call the process $P_T(X,W)$ the transformed process of $(X,W)$ with respect to $T$, and we call the map $P_T$ the process transformation associated with $T$.
\end{definition}

We remark that if   $T$ is a strong stochastic transformation and  $W$ is a Brownian motion, then  $P_T(W)=W$.

\begin{definition}\label{SDEtrasformation}
Let $T=(\Phi,B,\eta)$ be a stochastic transformation. If the pair $(\mu,\sigma)$ is a SDE on
$M$, we define $E_T(\mu,\sigma)=(E_T(\mu),E_T(\sigma))$ the SDE on $M'$ given by
\begin{eqnarray*}
E_T(\mu)&=&\left(\frac{1}{\eta} L(\Phi) \right) \circ \Phi^{-1}\\
E_T(\sigma)&=&\left( \frac{1}{\sqrt{\eta}} D(\Phi) \cdot \sigma \cdot
B^{-1}\right) \circ \Phi ^{-1}.
\end{eqnarray*}
We call the SDE $E_T(\mu,\sigma)$ the transformed SDE of $(\mu,\sigma)$ with respect to $T$, and we call the map $E_T$ the SDE transformation associated with $T$.
\end{definition}

\begin{theorem}\label{theorem_transformation}
Let $T$ be a stochastic transformation and let $(X,W)$ be a solution to the
SDE $(\mu,\sigma)$, then $P_T(X,W)$ is a solution to the SDE
$E_T(\mu,\sigma)$.
\end{theorem}
\begin{proof}
The proof is a simple combination of Propositions
\ref{proposition_space_transformation}, Proposition
\ref{proposition_time_transformation} and Proposition
\ref{proposition_Brownian_transformation}. Indeed if we first
apply time and Brownian transformations, i.e. Proposition \ref{proposition_time_transformation} and Proposition
\ref{proposition_Brownian_transformation}  (the order of these two
applications does not matter) and then we apply space
transformation, i.e. Proposition \ref{proposition_space_transformation}, we obtain the thesis.
${}\hfill$\end{proof}

\section{SDE Transformations: a geometric analysis}

\subsection{The geometric description of stochastic transformations}

Let us consider the group $G=SO(m) \times \mathbb{R}_+$, with the natural product given by $g_1 \cdot g_2=(A_1 \cdot A_2,\zeta_1 \zeta_2)$, where  $g_1=(A_1,\zeta_1) $ and $g_2=(A_2,\zeta_2) $.\\
Since the manifold $M \times G$ is a trivial principal bundle $\pi_M: M \times G \longrightarrow  M$ with structure group $G$,
we can consider the following action of the group $G$ on $M \times G$ leaving $M$ invariant
\begin{eqnarray*}
R_{M,h}: & M \times G &\longrightarrow  M \times G\\
& (x,g) & \longmapsto (x,g \cdot h).
\end{eqnarray*}
\begin{definition}\label{definition_isomorphism} Given two (trivial) principal bundles $M \times G$ and $ M' \times G$, an isomorphism $F:M \times G \rightarrow M' \times G$ is a diffeomorphism that preserves the structure of principal bundle of $M \times G$ and of $M' \times G$. This means that there exists a diffeomorphism $\Phi: M \rightarrow M'$ such that
$$F \circ \pi_{M'} = \pi_M \circ \Phi,$$
and, for any $h \in G$,
$$F \circ R_{M,h}= R_{M',h} \circ F.$$
We denote by $\Iso(M \times G, M' \times G)$ the set of isomorphisms between $M \times G$ and $M' \times G$.
\end{definition}
It is easy to check that the previous definition ensures that any $F\in \Iso(M \times G, M' \times G)$ is completely determined by its value on $(x, e)$, (where $e$ is the unit of $G$) i.e. there is a one-to-one correspondence between $F$ and the pair $F(x,e)=(\Phi(x),g)$. Therefore there exists a natural identification between a stochastic transformation $T=(\Phi,B,\eta) \in S_{m}(M,M')$ and the isomorphism $F_T$ defined by
$$F_T(x,g)=(\Phi(x),(B(x),\eta(x)) \cdot g)$$
and the set $S_m(M,M')$ inherits the properties of the set $\Iso(M \times G, M' \times G)$. In particular, the natural composition of an element of $\Iso(M \times G, M' \times G)$ with an element of $\Iso(M' \times G, M'' \times G)$ to give an element  of $\Iso(M \times G, M'' \times G)$ ensures the existence of a natural composition law between elements of $S_m(M,M')$ and of $S_m(M',M'')$. If $T=(\Phi,B,\eta) \in S_m(M,M')$ and $\tilde{T}=(\tilde{\Phi},\tilde{B},\tilde{\eta}) \in S_m(M',M'')$, we have
$$\tilde{T} \circ T=(\tilde{\Phi} \circ \Phi,(\tilde{B} \circ \Phi) \cdot B, (\tilde{\eta} \circ \Phi)\eta).$$
Moreover, since $\Iso(M \times G,M' \times G)$ is a subset of the diffeomorphism between $M \times G$ and $M' \times G$, if $T \in S_m(M,M')$ we can define its inverse $T^{-1} \in S_m(M',M)$ as
$$T^{-1}=(\Phi^{-1},(B \circ \Phi^{-1})^{-1},(\eta \circ \Phi^{-1})^{-1}).$$
The set $S_m(M):=S_m(M,M)$ is a group with respect to the composition $\circ$ and the identification
 of $S_m(M)$ with $\Iso(M \times G, M \times G)$ (which is a closed subgroup of the group of diffeomorphism of $M \times G$) suggests to consider  the corresponding Lie algebra $V_m(M)$.\\
 For later use, in the following we provide a description of the elements of $V_m(M)$.

Given a one parameter group  $T_a=(\Phi_a,B_a,\eta_a) \in S_m(M)$,  there exist a vector field $Y$ on $M$, a smooth function $C:M \rightarrow so(m)$ (where $so(m)$ is the Lie algebra of antisymmetric matrices), and a smooth function $\tau:M \rightarrow \mathbb{R}$ such that
\begin{equation}\label{equation_infinitesimal_SDE1}\begin{array}{ccc}
Y(x)&:=&\partial_a(\Phi_a(x))|_{a=0}\\
C(x)&:=&\partial_a(B_a(x))|_{a=0}\\
\tau(x)&:=&\partial_a(\eta_a(x))|_{a=0}.
\end{array}
\end{equation}
Conversely, considering  $Y,C,\tau$  as above, the one parameter solution $(\Phi_a,B_a,\eta_a)$ to the equations
\begin{equation}\label{equation_infinitesimal_SDE2}\begin{array}{lcr}
\partial_a(\Phi_a(x))&=&Y(\Phi_a(x))\\
\partial_a(B_a(x))&=&C(\Phi_a(x)) \cdot B_a(x)\\
\partial_a(\eta_a(x))&=&\tau(\Phi_a(x))\eta_a(x).
\end{array}
\end{equation}
with initial condition $\Phi_0=id_M$, $B_0=I_m$ and $\eta_0=1$, is a one parameter group in $S_m(M)$. For this reason we identify the elements of $V_m(M)$ with the triples $(Y,C,\tau)$. \\
\begin{definition}
A triad $V=(Y,C,\tau)\in V_m(M)$, where  $Y$ is a vector field on $M$ and  $C:M \rightarrow so(m)$  and  $\tau:M \rightarrow \mathbb{R}$  are smooth functions, is an  infinitesimal stochastic transformations. If $V$ is of the form  $V=(Y,0,0)$  we call $V$ a strong infinitesimal stochastic transformations, as the corresponding one-parameter group is a group of strong stochastic transformations.
\end{definition}
Since $V_m(M)$ is a sub-Lie algebra of the set of vector fields on $M \times G$, the standard Lie brackets between vector fields on $M \times G$ induces some Lie brackets on $V_m(M)$. Indeed, if $V_1=(Y_1,C_1,\tau_1),V_2=(Y_2,C_2,\tau_2) \in V_m(M)$ are two infinitesimal stochastic transformations, we have
\begin{equation}\label{equation_infinitesimal_SDE5}
[V_1,V_2]=([Y_1,Y_2],Y_1(C_2)-Y_2(C_1)-\{C_1,C_2\},Y_1(\tau_2)-Y_2(\tau_1)),
\end{equation}
where $\{\cdot,\cdot\}$ denotes the usual commutator between matrices.\\
Furthermore the identification of  $T=(\Phi,B,\eta) \in S_m(M,M')$  with  $F_T\in \Iso(M \times G, M' \times G)$ allows us to define  the push-forward $T_*(V)$  of $V \in V_m(M)$ as
\begin{equation}\label{equation_infinitesimal_SDE4}
((D(\Phi) \cdot Y) \circ \Phi ^{-1} ,(B \cdot C \cdot B^{-1}+Y(B) \cdot B^{-1}) \circ \Phi^{-1},(\tau+Y(\eta)\eta^{-1})\circ \Phi^{-1}).
\end{equation}
Analogously, given $V' \in V_m(M')$, we can consider  the pull-back of $V'$  defined as $T^*(V')=(T^{-1})_*(V')$.

The following Theorem shows that any Lie algebra of general infinitesimal stochastic transformations satisfying a non-degeneracy condition, can be locally transformed, by action of the push-forward of a suitable stochastic transformation $T \in S_m(M)$,  into a Lie algebra of strong infinitesimal stochastic transformations.

\begin{theorem}\label{theorem_infinitesimal_SDE1}
Let $K=\spann \{V_1,...,V_k\}$ be a Lie algebra of $V_m(M)$ and
let $x_0 \in M$ be such that $Y_1(x_0),...,Y_k(x_0)$ are linearly
independent (where $V_i=(Y_i,C_i,\tau_i)$). Then there exist an
open neighborhood $U$ of $x_0$ and a stochastic transformation $T \in
S_m(U)$ of the form $T=(id_U,B,\eta)$ such that
$T_*(V_1),...,T_*(V_k)$ are strong infinitesimal stochastic
transformations in $V_m(U)$. Furthermore the smooth functions
$B,\eta$ are solutions to the equations
\begin{eqnarray*}
Y_i(B)&=&-B \cdot C_i\\
Y_i(\eta)&=&-\tau_i \eta,
\end{eqnarray*}
for $i=1,...,k$.
\end{theorem}
\begin{proof}
By equation \refeqn{equation_infinitesimal_SDE4} with  $T=(id_U,B,\eta)$ we have
$$T_*(V_i)=(Y_i,Y_i(B)\cdot B^{-1}+B \cdot C_i \cdot B^{-1},Y_i(\eta)\eta^{-1}+\tau_i),$$
and $T_*(V_i)$ is a strong infinitesimal stochastic transformation if and only if
\begin{eqnarray}
Y_i(B)\cdot B^{-1}+B \cdot C_i \cdot B^{-1}&=&0 \label{equation_matrix1}\\
Y_i(\eta)\eta^{-1}+\tau_i&=&0.\label{equation_matrix2}
\end{eqnarray}
Denote by $L_i,N_i$ the linear operators on $\matr(m,m)$-valued and $\mathbb{R}_+$-valued smooth functions respectively such that
\begin{eqnarray*}
L_i(B):&=&Y_i(B)+B \cdot C_i=(Y_i+R_{C_i})(B)\\
N_i(B):&=&Y_i(\eta)+ \eta \tau_i=(Y_i+R_{\tau_i})(\eta),
\end{eqnarray*}
where $R_{C_i},R_{\tau_i}$ are the operators of right multiplication.
A sufficient condition for the existence of a non-trivial solution to equations \refeqn{equation_matrix1} and \refeqn{equation_matrix2}, is that there exist some real constants $c_{i,j}^k,d_{i,j}^k$ such that
\begin{eqnarray}
L_i L_j- L_J L_i&=&\sum_k c^k_{i,j} L_k\label{equation_matrix3}\\
N_i N_j-N_j N_i &=&\sum_k d^k_{i,j} N_k.\label{equation_matrix4}
\end{eqnarray}
A simple computation shows   that
\begin{eqnarray}
L_i L_j- L_J L_i&=&[Y_i,Y_j]+R_{Y_i(C_j)-Y_j(C_i)-\{C_i,C_j\}}\label{equation_matrix5}\\
&=&[Y_i,Y_j]+R_{Y_i(\tau_j)-Y_j(\tau_i)}.\label{equation_matrix6}
\end{eqnarray}
Since $V_i=(Y_i,C_i,\tau_i)$ form a Lie algebra, there exist some constants $f_{i,j}^k$ such that
\begin{eqnarray*}
[V_i,V_j]&=&([Y_i,Y_j],Y_i(C_j)-Y_j(C_i)-\{C_i,C_j\},Y_i(\tau_j)-Y_j(\tau_i))\\
&=&\left(\sum_k f_{i,j}^k Y_k, \sum_k f_{i,j}^k C_k, \sum_k f^k_{i,j} \tau_k\right).
\end{eqnarray*}
Comparing the last equality with equations \refeqn{equation_matrix5} and \refeqn{equation_matrix6} we find   equations \refeqn{equation_matrix3} and \refeqn{equation_matrix4} and this completes the proof.
${}\hfill$\end{proof}

\subsection{Probabilistic foundation of the geometric description}\label{section_probability_meaning}
In this subsection we show how the identification of stochastic transformations with the isomorphisms of suitable trivial principal bundles and the resulting natural definition of composition of stochastic transformations has a deep probabilistic counterpart in terms of SDE and process transformations introduced in Definition \ref{processtrasformation} and \ref{SDEtrasformation}.

\begin{theorem}\label{proposition_transformation}
If $T \in S_m(M,M')$ and $T' \in S_m(M',M'')$ are two stochastic
transformations, then
\begin{eqnarray*}
P_{T'} \circ P_T&=&P_{T' \circ T}\\
E_{T'} \circ E_T&=&E_{T' \circ T}.
\end{eqnarray*}

\end{theorem}
\begin{proof}
We have to prove that, for any stochastic process $(X,W)$ and for any SDE $(\mu,\sigma)$, we have
\begin{eqnarray*}
P_{T'}(P_T(X,W))=P_{T' \circ T}(X,W)\\
E_{T'}(E_T(\mu,\sigma))=E_{T' \circ T}(\mu,\sigma).
\end{eqnarray*}
We prove the proposition for $X$ in the pair $(X,W)$. The proof for $W$ and $(\mu,\sigma)$
is similar. Suppose $T=(\Phi,B,\eta)$ and $T'=(\Phi',B',\eta')$, and let us put
$\beta_t=\int_0^t{\eta(X_s)ds}$, $\beta'_t=\int_0^t{\eta'(\Phi_{\eta}(X_s))ds}$. Then we have that
\begin{eqnarray*}
P_{T'}(P_{T}(X))&=&\Phi'(H_{\eta'}(\Phi(H_{\eta}(X))))\\
&=&\Phi'(H_{\beta'}(\Phi(H_{\beta}(X))))\\
&=&\Phi' \circ \Phi(H_{\beta'}(H_{\beta}(X))).
\end{eqnarray*}
where we have used the definition $H_{\eta}:=H_{\beta}$ given in Section 2.2.
We want to calculate the composite random time change $H_{\beta'} \circ H_{\beta}$. Let $\alpha_t, \alpha'_t$ be the inverses of the processes $\beta_t,\beta'_t$. If $Y$ is any continuous process
$$H_{\beta'}(H_{\beta}(Y))_t= (H_{\beta}(Y))_{\alpha'_t}=Y_{\alpha_{\alpha'_t}}.$$
Since the inverse of $\alpha_{\alpha_t}$ is $\beta'_{\beta_t}$ we have that
$$H_{\beta'} \circ H_{\beta}=H_{\beta' \circ \beta}.$$
If we compute the density of the time change $\beta'_{\beta_t}$, we find
\begin{eqnarray*}
\beta'_{\beta_t}&=& \int_0^{\beta_t}{\eta'(\Phi_{\eta}(X)_s)ds}\\
&=&\int_0^t{H_{\alpha}(\eta'(\Phi_{\eta}(X))_s d\beta_s}\\
&=&\int_0^t{\eta'(\Phi(H_{\alpha}(H_{\beta}(X)))_s \eta(X_s)ds}\\
&=&\int_0^t{(\eta'\circ \Phi)(X_s)\eta(X_s)ds},
\end{eqnarray*}
and we have
$$H_{\beta'} \circ H_{\beta}(X)=H_{(\eta' \circ \Phi) \eta}(X).$$
Hence
$$P_{T'}(P_{T}(X))=\Phi' \circ \Phi(H_{\beta'}(H_{\beta}(X)))=\Phi' \circ \Phi(H_{(\eta' \circ \Phi) \eta}(X))=P_{T' \circ T}(X).$$
by definition of composition between stochastic transformations given in Section 3.1.
${}\hfill$\end{proof}

\section{Symmetries of a SDE}

\subsection{Finite symmetries}
In analogy with the usual distinction between  strong and  weak
solutions to a SDE we give the following

\begin{definition}\label{definition_symmetry_SDE}
A strong stochastic transformation $T \in SS_m(M)$ is  a strong (finite)
symmetry  of the SDE $(\mu,\sigma)$
if for any  solution $(X,W)$ to  $(\mu,\sigma)$ the transformed process $P_T(X,W)=(P_T(X),W)$
(Brownian motion is unchanged) is also
a solution to  $(\mu,\sigma)$. A stochastic transformation $T
\in S_m(M)$ is called a weak (finite) symmetry
of the SDE $(\mu,\sigma)$ if for any  solution $(X,W)$ to
$(\mu,\sigma)$, the generic transformed process $P_T(X,W):=(P_T(X),P_T(W))$
is also a solution to
$(\mu,\sigma)$ (Brownian motion is changed).
\end{definition}

\begin{theorem}\label{theorem_SDE3}
A strong stochastic transformation $T=(\Phi,B,\eta) \in SS_m(M)$ is a strong
symmetry of a SDE $(\mu,\sigma)$ if and only if
\begin{eqnarray}
L(\Phi) \circ \Phi^{-1}&=&\mu\label{equation_finite1}\\
(D(\Phi) \cdot \sigma )\circ \Phi^{-1}&=& \sigma.\label{equation_finite2}
\end{eqnarray}
A stochastic transformation $T \in S_m(M)$ is a weak symmetry of a SDE
$(\mu,\sigma)$ if and only if
\begin{eqnarray}
\left(\frac{1}{\eta} L(\Phi) \right) \circ \Phi^{-1}&=&\mu \label{equation_finite3}\\
\left( \frac{1}{\sqrt{\eta}} D(\Phi) \cdot \sigma \cdot B^{-1}\right) \circ \Phi^{-1}&=&\sigma.\label{equation_finite4}
\end{eqnarray}
\end{theorem}
\begin{proof}
We prove the proposition for weak symmetries. For strong symmetries the proof is similar.\\
If a stochastic transformation $T$ satisfies equations \refeqn{equation_finite3} and \refeqn{equation_finite4} then $E_T(\mu,\sigma)=(\mu,\sigma)$. We have to prove that $T$ is a  symmetry of the SDE $(\mu,\sigma)$.
Let $(X,W)$ be a solution to  $(\mu,\sigma)$:   Theorem \ref{theorem_transformation} ensures that $P_T(X,W)$ is a solution to $E_T(\mu,\sigma)=(\mu,\sigma)$.\\
Conversely, if for any solution $(X,W)$  to $(\mu,\sigma)$ also $P_T(X,W)$ is a solution to $(\mu,\sigma)$, since the coefficients $(\mu,\sigma)$ are smooth on $M$, for any $x_0 \in M$ there exists a solution $(X^{x_0},W^{x_0})$ (until the stopping time $\tau^{x_0}$ with $\mathbb{P}(\tau^{x_0}>0)=1$) such that $\mathbb{P}(X_0=x_0)=1$ (see Chapter IV, Theorem 2.3 in \cite{Watanabe}). Moreover, being $T \in T_m(M)$  a symmetry of the SDE $(\mu,\sigma)$, also $P_T(X^{x_0},W^{x_0})$ is a solution to the SDE $(\mu,\sigma)$ and, by Theorem \ref{theorem_transformation}, $P_T(X^{x_0},W^{x_0})$ is also a solution to $E_T(\mu,\sigma)$. This means that, for any $t \in \mathbb{R}$, the following equations hold
\begin{eqnarray*}
P_T(X^{x_0})^i_{t \wedge \tau^{x_0}}-P_T(X^{x_0})^i_0&=&\int_0^{t \wedge \tau^{x_0}}{\mu^i(P_T(X^{x_0})_{s})ds}\\
&&+\sum_{\alpha}\int_0^{t \wedge \tau^{x_0}}{\sigma^i_{\alpha}(P_T(X^{x_0})_{s})dP_T(W^{x_0})^{\alpha}_s}.\\
P_T(X^{x_0})^i_{t \wedge \tau^{x_0}}-P_T(X^{x_0})^i_0&=&\int_0^{t \wedge \tau^{x_0}}{E_T(\mu)^i(P_T(X^{x_0})_{s})ds}\\
&&+\sum_{\alpha}\int_0^{t \wedge \tau^{x_0}}{E_T(\sigma)^i_{\alpha}(P_T(X^{x_0})_{s})dP_T(W^{x_0})^{\alpha}_s}.\\
\end{eqnarray*}
This means that the process
\begin{eqnarray*}
Z^i_t&=&\int_0^{t \wedge \tau^{x_0}}{(E_T(\mu)^i(P_T(X^{x_0})_{s})-\mu^i(P_T(X^{x_0})_{s}))ds}\\
&&+\sum_{\alpha}\int_0^{t \wedge \tau^{x_0}}{(E_T(\sigma)_{\alpha}^i(P_T(X^{x_0})_{s})-\sigma^i_{\alpha}(P_T(X^{x_0})_{s}))dP_T(W^{x_0})^{\alpha}_s}
\end{eqnarray*}
is identically zero. As a consequence, also the quadratic variation $[Z^i,Z^i]$ is identically zero, and, by Theorem \ref{theorem_quadratic_variation}, we obtain
$$0=\sum_{\alpha}\int_0^{t \wedge \tau^{x_0}}{(E_T(\sigma)^i_{\alpha}(P_T(X^{x_0})_{s})-\sigma^i_{\alpha}(P_T(X^{x_0})_{s}))^2dt}.$$
The previous reasoning implies also that
$$\int_0^{t \wedge \tau^{x_0}}{(E_T(\mu)^i(P_T(X^{x_0})_{s})-\mu^i(P_T(X^{x_0})_{s}))^2ds}=0.$$
Since $P_T(X)_s$ is a continuous process and $E_T(\mu),E_T(\sigma),\mu,\sigma$ are smooth functions, we necessarily have that
\begin{eqnarray*}
E_T(\mu)^i(P_T(X^{x_0})_{t \wedge \tau^{x_0}})-\mu^i(P_T(X^{x_0})_{t \wedge \tau^{x_0}})&=&0\\
(E_T(\sigma)^i_{\alpha}(P_T(X^{x_0})_{t \wedge \tau^{x_0}})-\sigma^i_{\alpha}(P_T(X^{x_0})_{t \wedge \tau^{x_0}}))^2&=&0.
\end{eqnarray*}
Taking $t=0$ and recalling that $\mathbb{P}(X_0=x_0)=\mathbb{P}(P_T(X)_0=\Phi(x_0))=1$ (being $P_T(X)_0=\Phi(X_0)$), we finally have that
\begin{eqnarray*}
E_T(\mu)^i(\Phi(x_0))&=&\mu^i(\Phi(x_0)) \\
E_T(\sigma)^i_{\alpha}(\Phi(x_0))&=&\sigma^i_{\alpha}(\Phi(x_0)).
\end{eqnarray*}
Since $x_0$ is a generic point of $M$ and $\Phi$ is a generic diffeomorphism we obtain $E_T(\mu,\sigma)=(\mu,\sigma)$, that is equivalent to equations \refeqn{equation_finite3} and \refeqn{equation_finite4}.
${}\hfill$\end{proof} \\

\subsection{Infinitesimal symmetries }

\begin{definition}\label{definition_symmetry_infinitesimal_SDE}
An infinitesimal stochastic transformation $V$ generating a one
parameter group $T_a$ is called  strong (or  weak) infinitesimal
symmetry of the SDE $(\mu,\sigma)$ if $T_a$ is a finite strong
(or weak) symmetry of the SDE $(\mu,\sigma)$.
\end{definition}

The following theorem provides the {\it determining equations} for the infinitesimal symmetries of a SDE.
They differ from those given in \cite{Gaeta1} and in \cite{Grigoriev} for the presence of the antisymmetric matrix $C$ and the smooth function $\tau$. \\
In order to avoid the use of many indices we introduce the following notation: if $A:M \rightarrow \mathbb{R}^n$ and $B:M \rightarrow \matr(n,m)$,
we denote by $[A,B]$ the smooth function $[A,B]:M \rightarrow
\matr(n,m)$ such that
$$[A,B]^i_j=A^k\partial_k(B^i_j)-B^k_j\partial_k(A^i),$$
where we use Einstein summation convention.\\
The bracket $[\cdot,\cdot]$ satisfies the following  properties:
$$[A,[C,B]]=[[A,C],B]+[C,[A,B]]$$
$$[A,B \cdot D]=[A,B] \cdot D+ B \cdot A(D).$$
In the particular case $B:M \rightarrow \mathbb{R}^n$, the expression of $[A,B]$ coincides with the usual Lie bracket between the vector fields $A,B$.

\begin{theorem}\label{theorem_infinitesimal_symmetry1}
Let $V=(Y,C,\tau)$ be an infinitesimal stochastic transformation. Then $V$
is an infinitesimal  symmetry of the SDE $(\mu,\sigma)$ if and only if
$Y$ generates a one parameter group on $M$ and
\begin{eqnarray}
Y(\mu)-L(Y)+\tau \mu&=&0 \label{equation_symmetry1.1}\\
\left[ Y,\sigma \right]+\frac{1}{2} \tau \sigma+\sigma \cdot C&=&0. \label{equation_symmetry1.2}
\end{eqnarray}
\end{theorem}
\begin{proof}
Let $V$ be an infinitesimal symmetry of $(\mu,\sigma)$ and let
$T_a=(\Phi_a,B_a,\eta_a)$ be the one-parameter group generated by
$V$. By Theorem \ref{theorem_SDE3}, we have that
\begin{eqnarray*}
\left(\frac{1}{\eta_a} L(\Phi_a) \right) \circ \Phi_{-a}&=&\mu\\
\left( \frac{1}{\sqrt{\eta_a}} D(\Phi_a) \cdot \sigma \cdot B_a^{-1}\right) \circ \Phi_{-a}&=&\sigma.
\end{eqnarray*}
If we compute the derivatives with respect to $a$ of the previous expressions
and take $a=0$ we obtain equations \refeqn{equation_symmetry1.1} and \refeqn{equation_symmetry1.2}.\\
Conversely suppose that equations \refeqn{equation_symmetry1.1} and \refeqn{equation_symmetry1.2} hold. If we define  $\mu_a$ and  $\sigma_a$ as
\begin{eqnarray}
\mu_a&=&\left(\frac{1}{\eta_a} L(\Phi_a) \right) \circ \Phi_{-a} \label{equation_symmetry5}\\
\sigma_a&=&\left( \frac{1}{\sqrt{\eta_a}} D(\Phi_a) \cdot \sigma \cdot B_a^{-1}\right) \circ \Phi_{-a},\label{equation_symmetry6}
\end{eqnarray}
the functions $\mu_a,\sigma_a$ solve the following first order partial differential equations
\begin{eqnarray}
\partial_a(\mu_a)&=&[Y,\mu_a]+A(\sigma_a,Y)+\tau \mu_a \label{equation_symmetry7}\\
\partial_a(\sigma_a)&=&\left[ Y,\sigma_a \right]+\frac{1}{2} \tau \sigma_a+\sigma_a \cdot C, \label{equation_symmetry8}
\end{eqnarray}
where
$$A(\sigma_a,Y)^i=\sum_{\alpha} \sigma^k_{\alpha,a}\sigma^h_{\alpha,a}\partial_{hk}(Y^i).$$

 If we consider
$\tilde{\sigma}_a=\sigma_a \circ \Phi_a$ and
$\tilde{\mu}_a=\mu_a \circ \Phi_a$,  equations
\refeqn{equation_symmetry7} and \refeqn{equation_symmetry8} become
\begin{eqnarray*}
\partial_a(\tilde{\mu}_a)&=&-\tilde{\mu}_a(Y)+A(\tilde{\sigma}_a,Y)+\tau \tilde{\mu}_a\\
\partial_a(\tilde{\sigma}_a)&=&-\sigma_a(Y)+\frac{1}{2} \tau \tilde{\sigma}_a+\tilde{\sigma}_a \cdot C
\end{eqnarray*}
that are, for $x$ fixed,  ordinary differential
equations in $a$ admitting a unique solution
for any initial condition $(\mu_0, \sigma_0)$.
As a consequence, when
\begin{equation}\label{equation_symmetry4}\begin{array}{ccc}
[Y,\mu_0]+A(\sigma_0,Y)+\tau \mu_0&=&0\\
\left[ Y,\sigma_0 \right]+\frac{1}{2} \tau \sigma_0+\sigma_0 \cdot C&=&0,
\end{array}
\end{equation}
we  have $\sigma_a=\sigma_0$ and $\mu_a=\mu_0$ for any $a$ and
\refeqn{equation_symmetry5} and \refeqn{equation_symmetry6} ensure that $T_a$ is a symmetry of  $(\mu,\sigma)$.
${}\hfill$\end{proof}\\

\begin{definition}
An infinitesimal stochastic transformation $V \in V_m(M)$ is a general
infinitesimal symmetry of the SDE $(\mu,\sigma)$ if satisfies
the determining equations \refeqn{equation_symmetry1.1} and \refeqn{equation_symmetry1.2}.
\end{definition}
In order to prove that  the Lie bracket of two general infinitesimal symmetries of a SDE is  a general  infinitesimal symmetry
of the same SDE we need the following technical Lemma:

\begin{lemma}\label{lemmautile}
Given a general infinitesimal symmetry  $(Y,C,\tau)$  of the SDE $(\mu,\sigma)$,  for any smooth function $f \in \cinf(M)$ we have
$$Y(L(f))-L(Y(f))=-\tau L(f),$$
where $L$ is the second order differential operator associated  with  $(\mu,\sigma)$.
\end{lemma}
\begin{proof}
Given $Y=Y^i\partial_{i}$ and $L=A^{ij}\partial_{ij}+\mu^i \partial_{i}$, we can write
\begin{eqnarray*}
Y(L(f))-L(Y(f))&=&Y^i\partial_{i}(A^{jk}\partial_{jk}(f)+\mu^j \partial_{j}(f))\\
&&-(A^{jk}\partial_{jk}(Y^i \partial_{i}(f))+\mu^j \partial_{j}(Y^i \partial_{i}(f)))\\
&=&(Y^i\partial_{i}(A^{jk})-A^{ik}\partial_{i}(Y^j)-A^{ji}\partial_{i}(Y^k))\partial_{jk}(f)\\
&&+(Y^i\partial_{i}(\mu^j)-A^{ik}\partial_{ik}(Y^j)-\mu^i \partial_{i}(Y^j))\partial_{j}(f)
\end{eqnarray*}
and the thesis of the Lemma reads
\begin{eqnarray}
(Y^i\partial_{i}(A^{jk})-A^{ik}\partial_{i}(Y^j)-A^{ji}\partial_{i}(Y^k))&=&-\tau (A^{jk})\label{equation_commutator1}\\
(Y^i\partial_{i}(\mu^j)-A^{ik}\partial_{ik}(Y^j)-\mu^i \partial_{i}(Y^j))&=&-\tau \mu^j. \label{equation_commutator2}
\end{eqnarray}
Equation \refeqn{equation_commutator2} can be written in the following way
$$Y(\mu)-L(Y)=-\tau \mu$$
and, denoting by $A$  the symmetric matrix of component $A^{ij}$, equation \refeqn{equation_commutator1} can be rewritten as follows
\begin{equation}\label{equation_commutator3}
Y(A)-D(Y) \cdot A-A \cdot D(Y)^T=-\tau A.
\end{equation}
Since $(Y,C,\tau)$ is a symmetry of the SDE $(\mu,\sigma)$, by equation \refeqn{equation_symmetry1.2} we have
\begin{equation}\label{equation_commutator4}\begin{array}{lcl}
[Y,\sigma] \cdot \sigma^T&=&-\frac{1}{2} \tau \sigma \cdot \sigma^T- \sigma \cdot  C \cdot \sigma^T\\
&=&- \tau A - \sigma \cdot C \cdot \sigma^T.
\end{array}
\end{equation}
If we sum equation \refeqn{equation_commutator4} with its transposed, being  $C$ an antisymmetric matrix, we obtain
\begin{equation}\label{equation_commutator5}\begin{array}{lcl}
[Y,\sigma] \cdot \sigma^T+([Y,\sigma] \cdot \sigma^T)^T&=&-2 \tau A -\sigma \cdot C \cdot \sigma^T-\sigma \cdot C^T\cdot \sigma\\
&=&-2 \tau A.
\end{array}
\end{equation}
Furthermore, since for any function $F$
$$[Y,F]=Y(F)-D(Y) \cdot F,$$
we have that
\begin{equation}\label{equation_commutator6}\begin{array}{lcl}
[Y,\sigma] \cdot \sigma^T+([Y,\sigma] \cdot \sigma^T)^T&=& -2(D(Y) \cdot A)+Y(\sigma) \cdot \sigma^T+\\
&&- 2(D(Y) \cdot A)^T+\sigma \cdot Y(\sigma)^T\\
&=&2(Y(A)-D(Y) \cdot A-A \cdot D(Y)^T).
\end{array}
\end{equation}
Using equations \refeqn{equation_commutator5} and \refeqn{equation_commutator6}  we obtain \refeqn{equation_commutator3}.
${}\hfill$\end{proof}

\begin{proposition}\label{theorem_infinitesimal_symmetry2}
Let $V_1=(Y_1,C_1,\tau_1),V_2=(Y_2,C_2,\tau_2) \in V_m(M)$ be two
general infinitesimal symmetries of the SDE $(\mu,\sigma)$, then
$[V_1,V_2]$ is a general infinitesimal symmetry of $(\mu,\sigma)$.
\end{proposition}
\begin{proof}
We start by proving that condition  \refeqn{equation_symmetry1.1} holds for $[V_1,V_2]$ defined by equation \refeqn{equation_infinitesimal_SDE5}, i.e.
$$
[Y_1,Y_2](\mu)-L([Y_1,Y_2])+(Y_1(\tau_2)-Y_2(\tau_1))\mu=0
$$
If we rewrite the left-hand side of the previous equation as
\begin{eqnarray*}
& &Y_1Y_2(\mu)-Y_2Y_1(\mu)-L([Y_1,Y_2])+Y_1(\tau_2\mu)-\tau_2Y_1(\mu)-Y_2(\tau_1\mu)+\tau_1Y_2(\mu) \\
&=&Y_1(Y_2(\mu)+\tau_2\mu)-Y_2(Y_1(\mu)+\tau_1\mu)-L([Y_1,Y_2])-\tau_2Y_1(\mu)+\tau_1Y_2(\mu)\\
&=&Y_1(L(Y_2))-Y_2(L(Y_1))- L([Y_1,Y_2])-\tau_2Y_1(\mu)+\tau_1Y_2(\mu)\\
\end{eqnarray*}
and we use Lemma \ref{lemmautile}, we get
\begin{eqnarray*}
& & [Y_1,Y_2](\mu)-L([Y_1,Y_2])+(Y_1(\tau_2)-Y_2(\tau_1))\mu\\
& &\tau_1(Y_2(\mu)-L(Y_2))- \tau_2(Y_1(\mu)-L(Y_1))  \\
&=&-\tau_1\tau_2\mu+\tau_2\tau_1\mu =0.\\
\end{eqnarray*}
Moreover  we have to prove that also  condition  \refeqn{equation_symmetry1.2} holds for $[V_1,V_2]$, i.e.
$$
[[Y_1,Y_2],\sigma] +\frac{1}{2}(Y_1(\tau_2)-Y_2(\tau_1))\sigma-\sigma \cdot \{C_1,C_2\}+\sigma \cdot  Y_1(C_2)+\sigma \cdot Y_2(C_1)=0.
$$
By using the properties of the Lie bracket we have
\begin{eqnarray*}
[[Y_1,Y_2],\sigma]&=&[Y_1,[Y_2,\sigma]]-[Y_2,[Y_1,\sigma]]\\
&=&-[Y_1,\frac{1}{2}\tau_2 \sigma+\sigma \cdot C_2]+[Y_2,\frac{1}{2}\tau_1 \sigma+\sigma \cdot C_1]\\
&=&-\frac{1}{2}Y_1(\tau_2)\sigma-\frac{1}{2}\tau_2 [Y_1,\sigma]-[Y_1,\sigma] \cdot C_2-\sigma \cdot Y_1(C_2)\\
&&+\frac{1}{2}Y_2(\tau_1)\sigma+\frac{1}{2}\tau_1 [Y_2,\sigma]+[Y_2,\sigma]\cdot C_1+\sigma \cdot Y_2(C_1)\\
&=&-\frac{1}{2}(Y_1(\tau_2)-Y_2(\tau_1))\sigma+\sigma \cdot \{C_1,C_2\}-\sigma \cdot  Y_1(C_2)+\sigma \cdot Y_2(C_1),
\end{eqnarray*}
and this conclude the proof.
${}\hfill$\end{proof}

\begin{proposition}\label{proposition_infinitesimal_symmetry}
Let $V \in V_m(M)$ be an infinitesimal symmetry of the SDE
$(\mu,\sigma)$ and let $T \in S_m(M,M')$ be a stochastic transformation.
Then $T_*(V)$ is an infinitesimal symmetry of $E_T(\mu,\sigma)$.
\end{proposition}
\begin{proof}
Given a solution $(X',W')$ to $E_T(\mu,\sigma)$, Theorem
\ref{theorem_SDE3} ensure that $(X,W)=P_T^{-1}(X',W')$ is a solution to
$(\mu,\sigma)$. If $T_a$ denotes the one-parameter group
generated by the infinitesimal symmetry $V$ then $P_{T_a}(X,W)$ is a solution to $(\mu,\sigma)$. By
Theorem \ref{theorem_SDE3}, $P_T(P_{T_a}(X,W))$ is a solution to the SDE
$E_T(\mu,\sigma)$ and, by Theorem
\ref{proposition_transformation}, for any $(X',W')$
solution to $E_T(\mu,\sigma)$, the process  $P_{(T \circ T_a \circ T^{-1})}(X',W')$ is
a solution to $E_T(\mu,\sigma)$. Since  the generator of $T \circ T_a
\circ T^{-1}$ is $T_*(V)$ we conclude that  $T_*(V)$ is an infinitesimal
symmetry of $E_T(\mu,\sigma)$.
${}\hfill$\end{proof}

\begin{theorem}\label{theorem_strong_symmetry}
Let $V_1=(Y_1,C_1,\tau_1),...,V_k=(Y_k,C_k,\tau_k)$ be general infinitesimal symmetries of $(\mu,\sigma)$.
 If $x_0 \in M$ is such that $Y_1(x_0),...,Y_k(x_0)$ are linearly independent,  then there exist a neighborhood $U$ of $x_0$ and a stochastic transformation $T \in S_m(U,U')$ such that $T_*(V_i)$ are strong infinitesimal symmetries of $E_T(\mu,\sigma)$.
\end{theorem}
\begin{proof}
The theorem is an application of Theorem \ref{theorem_infinitesimal_SDE1} and Proposition \ref{proposition_infinitesimal_symmetry}.
${}\hfill$\end{proof}

\section{Applications}

\subsection{A stochastic perturbation of a symmetric ODE}

Let us consider the following SDE
\begin{equation}\label{equation_example21}
\left(\begin{array}{c}
dX_t\\
dY_t \end{array}
\right)=\left(\begin{array}{c}
\frac{X_t}{X_t^2+Y_t^2}\\
\frac{Y_t}{X_t^2+Y_t^2} \end{array}
\right)dt+\left(\begin{array}{cc}
1 & 0\\
0 & 1\end{array}
\right)\cdot \left(\begin{array}{c}
dW^1_t\\
dW^2_t \end{array}
\right)
\end{equation}
and  the infinitesimal stochastic transformations $(Y,C,\tau)$  of the form
\begin{equation}\label{stoch_transf_2dim}
Y=\left(\begin{array}{c}
f(x,y)\\
g(x,y)\end{array}
\right),C=\left(\begin{array}{cc}
0 & c(x,y)\\
-c(x,y) & 0 \end{array}
\right),\tau(x,y).
\end{equation}
The determining equations for the symmetries of the SDE \refeqn{equation_example21} are
\begin{eqnarray*}
&\partial_x(f)=\frac{1}{2}\tau&\\
&\partial_y(f)=c&\\
&\partial_x(g)=-c&\\
&\partial_y(g)=\frac{1}{2}\tau&\\
&\frac{1}{2}(\partial_{xx}(f)+\partial_{yy}(f))+\frac{x}{x^2+y^2}\partial_{x}(f)+\frac{y}{x^2+y^2}\partial_y(f)-\frac{y^2-x^2}{(x^2+y^2)^2}f+\frac{2xy}{(x^2+y^2)^2}g=\tau \frac{x}{x^2+y^2} &\\
&\frac{1}{2}(\partial_{xx}(g)+\partial_{yy}(g))+\frac{x}{x^2+y^2}\partial_{x}(g)+\frac{y}{x^2+y^2}\partial_y(g)+\frac{2xy}{(x^2+y^2)^2}f-\frac{x^2-y^2}{(x^2+y^2)^2}g=\tau \frac{y}{x^2+y^2},&\\
\end{eqnarray*}
which admit only two linearly independent solutions corresponding to the following weak symmetries $V_1,V_2$
\begin{eqnarray*}
V_1&=&\left(\left(\begin{array}{c}
x\\
y\end{array}
\right),\left(\begin{array}{cc}
0 & 0\\
0 & 0\end{array}
\right), 2\right)\\
V_2&=&\left(\left(\begin{array}{c}
y\\
-x\end{array}
\right),\left(\begin{array}{cc}
0 & 1\\
-1 & 0\end{array}
\right), 0\right).
\end{eqnarray*}
In order to apply  Theorem \ref{theorem_strong_symmetry} ensuring the existence of a stochastic transformation $T$ such that $T_*(V_1),T_*(V_2)$ are two strong symmetries of $E_T(\mu,\sigma)$, we look for  a finite stochastic transformation $T=(id_M,B,\eta)$ such that
\begin{eqnarray}
x \partial_x(B)+y\partial_y(B)&=& 0\label{equation_example2}\\
y\partial_x(B)-x \partial_y(B)&=& -B \cdot \left(\begin{array}{cc}
0 & 1\\
-1 & 0\end{array}
\right) \label{equation_example3}\\
x \partial_x(\eta)+y\partial_y(\eta)&=&-2 \eta \label{equation_example4}\\
y\partial_x(\eta)-x \partial_y(\eta)&=&0.\label{equation_example5}
\end{eqnarray}
We solve the previous equations in detail. Since the vector fields $Y_1=(x,y)^T$ and $Y_2=(y,-x)^T$  satisfy the hypothesis of Theorem \ref{theorem_strong_symmetry} only in $\mathbb{R}^2-\{(0,0)\}$,  we consider $M=\mathbb{R}^2 - \{(0,0)\}$. If we write $B:M \rightarrow SO(2)$ in the form
$$B=\left(\begin{array}{cc}
b_1 & -b_2\\
b_2 & b_1\end{array}
\right),$$
with $b_1,b_2:M \rightarrow \mathbb{R}$ and $b_1^2+b_2^2=1$, equations \refeqn{equation_example2} and \refeqn{equation_example3} become
\begin{eqnarray*}
\left(\begin{array}{cc}
Y_1(b_1) & -Y_1(b_2)\\
Y_1(b_2) & Y_1(b_1)\end{array}
\right)&=&\left(\begin{array}{cc}
0 & 0\\
0 & 0\end{array}
\right)\\
\left(\begin{array}{cc}
Y_2(b_1) & -Y_2(b_2)\\
Y_2(b_2) & Y_2(b_1)\end{array}
\right)&=&\left(\begin{array}{cc}
-b_2 & -b_1\\
b_1 & -b_2\end{array}
\right).
\end{eqnarray*}
The previous equations ensure that $b_1,b_2$ are only functions of $\frac{y}{x}$ and that furthermore
\begin{equation}\label{equation_example6}
Y_2(Y_2(b_1))=-b_1.
\end{equation}
Suppose that $b_1(x,y)=h(\frac{y}{x})$ and introduce $\xi=\frac{y}{x}$. Then  equation \refeqn{equation_example6} can be rewritten as
$$ (\xi^2+1)^2 \partial_{\xi\xi}(h)+(2\xi^3+2\xi)\partial_{\xi}(h)+h=0.$$
and admits the following general solution
$$h(\xi)=\frac{k_1\xi}{\sqrt{\xi^2+1}}+\frac{k_2}{\sqrt{\xi^2+1}},$$
where $k_1,k_2 \in \mathbb{R}$. Choosing $k_2=1,k_1=0$ we have
$$b_1=\frac{x}{\sqrt{x^2+y^2}}$$
and from the relations $b_1^2+b_2^2=1$ we obtain
$$b_2=\frac{y}{\sqrt{x^2+y^2}}$$
so  that the matrix $B$ is
$$B=\left(\begin{array}{cc}
\frac{x}{\sqrt{x^2+y^2}} & -\frac{y}{\sqrt{x^2+y^2}}\\
\frac{y}{\sqrt{x^2+y^2}} & \frac{x}{\sqrt{x^2+y^2}}\end{array}
\right).$$
The computations for $\eta$ are similar. Indeed from  equation \refeqn{equation_example5} we find that $\eta$ must be a function  of $r=\sqrt{x^2+y^2}$, that is $\eta(x,y)=k(r)$ and  equation \refeqn{equation_example4} reads
$$r \partial_{r}(k)=-2 k,$$
whose general solution is
$$k(r)=\frac{k_0}{r^2},$$
with  $k_0 \in \mathbb{R}$.  Choosing  $k_0=1$, we have
$$\eta(x,y)=\frac{1}{x^2+y^2}$$
and we can write  the transformed SDE $E_T(\mu,\sigma)$
\begin{eqnarray*}
E_T(\mu)&=&\left(\begin{array}{c}
x \\
y \end{array}
\right)\\
E_T(\sigma)&=&\left(\begin{array}{cc}
x & y\\
-y & x \end{array} \right).
\end{eqnarray*}
So the new equations
\begin{equation}\label{equation_example1}
\left(\begin{array}{c}
dX_t\\
dY_t \end{array}
\right)=\left(\begin{array}{c}
X_t\\
Y_t \end{array}
\right)dt+\left(\begin{array}{cc}
X_t & Y_t\\
-Y_t & X_t\end{array}
\right)\cdot \left(\begin{array}{c}
dW'^1_t\\
dW'^2_t \end{array}
\right).
\end{equation}
admit the two strong symmetries
\begin{eqnarray*}
T_*(V_1)&=&\left(\left(\begin{array}{c}
x\\
y\end{array}
\right),\left(\begin{array}{cc}
0 & 0\\
0 & 0\end{array}
\right), 0\right)\\
T_*(V_2)&=&\left(\left(\begin{array}{c}
y\\
-x\end{array}
\right),\left(\begin{array}{cc}
0 & 0\\
0 & 0\end{array}
\right), 0\right).
\end{eqnarray*}

\subsection{Two dimensional Brownian motion}\label{2_dimensional_Brownian}

Let us consider the following SDE on $\mathbb{R}^2$

$$\left(\begin{array}{c}
dX_t\\
dY_t \end{array}
\right)=\left(\begin{array}{c}
0\\
0 \end{array}
\right)dt+\left(\begin{array}{cc}
1 & 0\\
0 & 1\end{array}
\right)\cdot \left(\begin{array}{c}
dW^1_t\\
dW^2_t \end{array}
\right)$$
\noindent with $\mu=0$ and $\sigma=I_2$. The solution to the previous equation is obviously the two dimensional Brownian motion.
By Theorem \ref{theorem_infinitesimal_symmetry1}, the infinitesimal transformation $(Y,C,\tau)$  of the form \refeqn{stoch_transf_2dim} is a general symmetry if and only if the following equations hold
\begin{eqnarray*}
&\partial_x(f)=\frac{1}{2}\tau&\\
&\partial_y(f)=c&\\
&\partial_x(g)=-c&\\
&\partial_y(g)=\frac{1}{2}\tau&\\
&\frac{1}{2}(\partial_{xx}(f)+\partial_{yy}(f))=0&\\
&\frac{1}{2}(\partial_{xx}(g)+\partial_{yy}(g))=0.&
\end{eqnarray*}
These equations are satisfied if and only if
\begin{eqnarray}
\partial_x(f)&=&\partial_y(g)\label{equation_brownian1}\\
\partial_y(f)&=&-\partial_x(g)\label{equation_brownian2}\\
c&=&\partial_y(f)\label{equation_brownian3}\\
\tau&=&2\partial_x(f).\label{equation_brownian4}
\end{eqnarray}
Therefore the two dimensional Brownian motion admits an infinite number of general infinitesimal symmetries. Indeed $f,g$ have to satisfy equations \refeqn{equation_brownian1}, \refeqn{equation_brownian2} that are the well-known Cauchy-Riemann equations. It is interesting to remark that the symmetry approach introduced before allows us to recover the expected property for the two dimensional Brownian motion (see \cite{DeLara1}, Example 4.1 and \cite{Liao}, Example 4).

Let us now discuss the problem of determining the general infinitesimal symmetries of the two dimensional Brownian motion generating a one-parameter group of stochastic transformations. Since the functions $f$ and $g$ satisfy the Cauchy-Riemann equations, and we want consider  functions $f,g$   defined on the whole plane $\mathbb{R}^2$, then  the function $u=f+ig$ is an entire function. Denoting  by $z=x+iy$, the vector field $Y$ is the real part of the holomorphic vector field $Z=u(z)\partial_z$ on $\mathbb{C}$ and $Y$ generates a one-parameter group on $\mathbb{R}$ if and only if $Z$ generates a one-parameter group on $\mathbb{C}$. Since $u$ is an entire function, $Z$ generates a one-parameter group if and only if $u(z)$ is a linear function in $z$ and $f,g$ must be linear functions in $x,y$  (see also \cite{Liao}, Example 4). \\
Therefore the general infinitesimal  symmetries of Brownian motion generating a one-parameter group are
\begin{eqnarray*}
V_1&=&\left(\left(\begin{array}{c}
1\\
0\end{array}
\right),\left(\begin{array}{cc}
0 & 0\\
0 & 0\end{array}
\right), 0\right)\\
V_2&=&\left(\left(\begin{array}{c}
0\\
1\end{array}
\right),\left(\begin{array}{cc}
0 & 0\\
0 & 0\end{array}
\right), 0\right)\\
V_3&=&\left(\left(\begin{array}{c}
x\\
y\end{array}
\right),\left(\begin{array}{cc}
0 & 0\\
0 & 0\end{array}
\right), 2\right)\\
V_4&=&\left(\left(\begin{array}{c}
y\\
-x\end{array}
\right),\left(\begin{array}{cc}
0 & 1\\
-1 & 0\end{array}
\right), 0\right).\\
\end{eqnarray*}
The infinitesimal stochastic transformations $V_1$ and $V_2$ are the $x$ and $y$ translation respectively, $V_3$ is the dilatation and $V_4$ is the rotation around the coordinate origin. It is important to note that, although for $V_1,V_2,V_3$ the matrix $C$  is $0$, in the case of the transformation $V_4$ we have $C\not =0$. This circumstance shows that  the introduction of the anti-symmetric matrix $C$ is necessary in order to have the rotation as a symmetry for the two dimensional Brownian motion.

\end{document}